\documentclass{article}
\usepackage{arxiv}

\usepackage[utf8]{inputenc} 
\usepackage[T1]{fontenc}    
\usepackage{hyperref}       
\usepackage{url}            
\usepackage{booktabs}       
\usepackage{amsfonts}       
\usepackage{nicefrac}       
\usepackage[english]{babel}
\usepackage{ae}
\usepackage{eucal}
\usepackage[dvips]{graphicx}
\usepackage{epsfig}
\usepackage{graphicx}
\usepackage{amsfonts,amsthm}
\usepackage{amssymb}
\usepackage{a4}
\usepackage{apu}
\usepackage{xcolor}
\usepackage{amsmath}
\usepackage{mathtools}
\usepackage{multicol}
\usepackage{textcomp}
\usepackage[all]{xy}

\usepackage{dcolumn,amsthm}

\renewenvironment{proof}[1][Proof]{\noindent\textit{#1. } }{\hfill$\square$}

\newtheoremstyle{theorem}{6pt}{6pt}{\rm}{}{\sffamily}{ }{ }{}
\theoremstyle{theorem}

\newtheoremstyle{lemma}{6pt}{6pt}{\rm}{}{\sffamily}{ }{ }{}
\theoremstyle{lemma}

\newtheoremstyle{example}{6pt}{6pt}{\rm}{}{\sffamily}{ }{ }{}
\theoremstyle{example}

\newtheoremstyle{corollary}{6pt}{6pt}{\rm}{}{\sffamily}{ }{ }{}
\theoremstyle{corollary}

\newtheoremstyle{definition}{6pt}{6pt}{\rm}{}{\sffamily}{ }{ }{}
\theoremstyle{definition}

\newtheoremstyle{remark}{6pt}{6pt}{\rm}{}{\sffamily}{ }{ }{}
\theoremstyle{remark}

\newtheoremstyle{approximation}{6pt}{6pt}{\rm}{}{\sffamily}{ }{ }{}
\theoremstyle{approximation}

\newtheoremstyle{scheme}{6pt}{6pt}{\rm}{}{\sffamily}{ }{ }{}
\theoremstyle{scheme}

\usepackage{biblatex}
\usepackage{csquotes}

\title{On the restriction of various Laplace operators on submanifolds}

\author{
   Jukka Tuomela \\
   Department of Physics and Mathematics\\
   University of Eastern Finland\\
   P.O. box 111, FI-80101 Joensuu, Finland  \\
   \text{jukka.tuomela@uef.fi} \\
}

\begin{document}

\maketitle

\begin{abstract}
When considering Navier-Stokes equations on Riemannian manifolds one frequently encounters situations where the manifold is embedded in the ambient Euclidean space. In this context it is interesting to investigate what is the precise relationship of the diffusion operator in the ambient space to the diffusion operator on the manifold. The present paper gives a precise characterization of this situation for general  surfaces in three dimensional space. 
\end{abstract}




\section{Introduction}
In recent years there has been a growing interest on analysing Navier-Stokes systems on Riemannian manifolds, see for example \cite{chczdi,fang,kobayashi,mj} and many references therein. An important issue in the analysis has turned out to be the choice of an appropriate diffusion operator. As discussed in \cite{czubak,mj} different choices can lead to very different solutions. It appears that it is at present somewhat unclear how to choose a correct diffusion operator in a concrete physical model.

In the present article we will discuss a related problem. Often we have a three dimensional flow, but we would like to study its restriction to some invariant submanifold. For example in the large scale atmospheric flow we could be interested in the flow on some (fictitious) surface. Then whatever diffusion operator one chooses, one would like to see how the three dimensional operator on the ambient space and the corresponding two dimensional operator on the relevant surface are related. In fact it is sufficient to analyse the Bochner Laplacian because other diffusion operators can be expressed in terms of it as will be explained below. Hence our results apply equally well for all diffusion operators. 

In \cite{chczyo} this problem was solved in the case of ellipsoid, and the present paper can be understood as a generalization of it. Hence we refer to \cite{chczyo} and its references for more background on this problem. In this paper we compute the restriction of Bochner Laplacian (and hence of any diffusion operator) to an arbitrary surface in three dimensional space. In fact the ambient space need not be $\mathbb{R}^3$; it can be any three dimensional Riemannian manifold. This is because the technical tool which we use, namely adapted moving frames (see \cite{spivak3}), works in the same way also in this general context. We compute also the normal component of the Bochner Laplacian. As far as we know the normal component has not been considered previously, but one would think that it could also be of some interest. 

Our approach is thus differential geometric. There has also been articles from the functional analysis point of view. One of the first papers in this direction for curved surfaces is \cite{temzia}; see also \cite{fang,miura,prsiwi} for more recent works. Perhaps combining more closely differential geometric and functional analytic ideas could lead to a better understanding of the restriction problem. 

In section 2 we introduce the necessary tools of Riemannian geometry  and discuss how Navier-Stokes equations are formulated on Riemannian manifolds. In section 3 we consider the problem treated in \cite{chczyo} in some detail because analyzing it in our framework gives some insights about the general case. In section 4 various types of extension/restriction are discussed, and it is seen that by the moving frame approach one can easily prove the existence of various types of extensions. In section 5 the main results are stated and proved. The results show that in general the diffusion operator restricted to the submanifold is rather unrelated to the diffusion operator on the submanifold.  Our results, however, give  precise formulas for the restriction which makes it then possible to study under which conditions there is a reasonable relationship between them.

\section{Preliminaries}

\subsection{Riemannian geometry }
Let us first recall  some facts about Riemannian geometry    which are needed later on; for details we refer to \cite{spivak2,spivak3}. Einstein summation convention is used whenever convenient. 
Let $M$ be a smooth Riemannian manifold with metric $g$.  For a function $f\,:\,M\to\mathbb{R}$ its gradient is $\mathsf{grad}(f)=g^{ij}f_{;j}$ and the divergence of the vector field $u$ is $\mathsf{div}(u)=u^i_{;i}$. The Ricci tensor is denoted by $\mathsf{Ri}$.  In order to simplify the notations we use the standard convention that when the indices are raised or lowered the notation does not change. Hence $u$ can denote both covariant and contravariant vector and  for example in two dimensional case we can write $\mathsf{Ri}=\kappa g$ where $\kappa$ is the Gaussian curvature and also $\mathsf{Ri}\,v=\kappa \,v$ where $v$ is a vector field. If the clarity requires to specify the type precisely then we use the musical isomorphisms $\sharp$ and $\flat$.

Let $M_0$ be a three dimensional Riemannian manifold with Riemannian metric $g_0$ and covariant derivative $\overline{\nabla}$. Let then $M$ be a two dimensional manifold with an embedding $f\,:\,M\to M_0$ and as usual we identify $M$ and $f(M)$ when convenient. The Riemannian metric on $M$ is $g=f^\ast g_0$ and the covariant derivative is denoted by $\nabla$. 

Let us now introduce an adapted orthonormal frame $\{b^1,b^2,b^3\}$ in the neighborhood of $M$, i.e. $b^1$ and $b^2$ are tangent along $M$ and $b^3$ is the unit normal. Note that globally such a frame may not exist, but since all the analysis that follows is local we only need a local frame which always exists. For simplicity of notation we will not indicate the domain  where the frame is defined. 
Then the dual frame is given by $\theta^j=\flat b^j$ and the associated connection
forms are the one forms $\omega_{ij}$ which satisfy the equations
\begin{align*}
    \omega_{ij}=&-\omega_{ji}\\
    d\theta^i=&-\sum_k \omega_{ik}\wedge \theta^k \ .
\end{align*}
Note that  $  \omega_{13}(b^2)=\omega_{23}(b^1)$.  Similarly we have  also an orthonormal frame $\{b^1_z,b^2_z\}=\{f^\ast b^1,f^\ast b^2\}$,   $\theta_z^j=\flat b_z^j$ and the connection forms $\omega_{ij}^z$ on $M$. The covariant derivative and connection forms are related by the formula
\begin{equation}
  \overline{\nabla}_{b^k}b^j=\sum_i \omega_{ij}(b^k)b^i\ .
\label{kova}
\end{equation}
Then we have the curvature forms $\Omega_{ij}$ defined by
  \begin{equation}
      d\omega_{ij}=-\sum_k \omega_{ik}\wedge
    \omega_{kj}+\Omega_{ij}\ .
  \label{kaare}
  \end{equation}
As the name suggests $\Omega_{ij}=0$ if $M_0$ is flat. The volume form of $M_0$, denoted $\Xi_0$, is given by $\Xi_0=\theta^1\wedge \theta^2\wedge\theta^3$ and similarly $\Xi=\theta_z^1\wedge \theta_z^2$ is the volume form on $M$.
On $M$ we have $d\omega_{12}^z=\Omega_{12}^z=\kappa\, \Xi$ and hence $\Omega_{12}^z(b_z^1,b_z^2)=\kappa$. 

Let then $X$ and $Y$ be vector fields on $M_0$ which are tangent along $M$. Then one can write
\begin{equation}
      \overline \nabla_{X}Y=
  \nabla_{X}Y+S(X,Y)b^3\ .
\label{toinen-perus}
\end{equation}
These are known as Gauss formulas and the  symmetric tensor $S$ is the second fundamental form. If $X$ is a vector field tangent along $M$, then
 \begin{equation}
    \overline{\nabla}_{X}b^3=-S\,X= \omega_{13}(X)b^1+
  \omega_{23}(X)b^2\ .
\label{wein}
\end{equation}
The first equality is known as Weingarten's equation. 
It will be convenient to define the functions $t_{ij}$ by
 \begin{align*}
    t_{11}=&S(b^1,b^1)=-\omega_{13}(b^1)
    &   t_{22}=&S(b^2,b^2)=-\omega_{23}(b^2)\\
        t_{12}=&S(b^1,b^2)=-\omega_{13}(b^2)
        =-\omega_{23}(b^1)
          \ .
 \end{align*}
If we consider the equality  $ t_{ij}=S(b^i,b^j)$ then one thinks that the functions $t_{ij}$ are defined on $M$ and in this case one could also write $ t_{ij}=S(b_z^i,b_z^j)$  ; however, the equality $ t_{ij}=-\omega_{i3}(b^j)$ allows one to interpret $t_{ij}$ as functions in the neighborhood of $M$. It turns out that the adjugate of $S$ appears in the formulas below so we introduce
\[
 S_{\mathsf{adj}}=
 t_{22}\theta^1\otimes \theta^1-
 t_{12}\theta^1\otimes \theta^2-
 t_{12}\theta^2\otimes \theta^1+
 t_{11}\theta^2\otimes \theta^2\ .
\]

It is also convenient to introduce the following notations:
\begin{equation}
    \alpha_j=\omega_{j3}(b^3)
    \quad\mathrm{and}\quad
    \gamma_j=\omega_{12}(b^j)\ .
\label{merk}
\end{equation}
Again depending on the context these can be interpreted as functions on $M$ or in the neighborhood of $M$. 
 Then we have
\begin{equation}
     \kappa=t_{11}t_{22}-t_{12}^2
 \quad\mathrm{and}\quad
\mathcal{H}=\tfrac{1}{2}\,
\big(t_{11}+t_{22}\big)
\label{kaarevuudet}
\end{equation}
where $\mathcal{H}$ is the mean curvature. Note that the sign of $\mathcal{H}$ depends on the choice of orientation of the normal field. Finally we need the formula
\begin{equation}
  d\omega(X,Y)=X(\omega(Y))-Y(\omega(X))-\omega\big([X,Y]\big)
  \label{omega-kaava}
\end{equation}
where $\omega$ is a one form and $X$ and $Y$ are vector fields.

On some occasions we will refer to ''standard results in PDE''; this means that the corresponding statements can easily be found for example  in \cite{evans}. When computing concrete examples we have used the \emph{Differential Geometry} package of {\sc Maple}.\footnote{\textsf{https://www.maplesoft.com/}}

\subsection{Navier-Stokes equations}

 We can write the incompressible Navier-Stokes equations on Riemannian manifolds as follows:
 \begin{equation}
 \begin{aligned}
   & u_t-Au+\nabla_u u+\mathsf{grad}(p)=0\\
    &\mathsf{div}(u)=0
\end{aligned}
\label{ns}
 \end{equation}
Here $A$ is some diffusion operator, and as discussed in \cite{czubak,mj}, there are basically three candidates for $A$; namely Bochner Laplacian, Hodge Laplacian and symmetric Laplacian. Let us first set
\[
   \Sigma u     =g^{ki}u^j_{;i}+g^{ij}u^k_{;i}\ .
\]
Then the various Laplacians are
\begin{align*}
    \Delta_Bu=&g^{ij}u^k_{;ij}\\
    \Delta_H u=&-\sharp (d\delta+\delta d)\flat u=
       \Delta_B u-\mathsf{Ri}(u)\\
       Lu=&\mathsf{div}(\Sigma u)=
       \Delta_B u +\mathsf{grad}(\mathsf{div}(u))+\mathsf{Ri}(u)
\end{align*}
Here  $\delta$ is the codifferential and we recall that for vector fields $\mathsf{div}(u)=-\delta\flat u$. It is seen that for flat manifolds Hodge Laplacian is the same as Bochner Laplacian, and if in addition $u$ is divergence free, then the symmetric Laplacian is also equal to them. Note that for divergence free vector fields all diffusion operators can be written as
\[
  Au=\Delta_B u+\beta\,\mathsf{Ri}(u)
\]
where $\beta=\pm 1$ or $\beta=0$. Let us then recall that
\begin{itemize}
    \item[(i)] $u$ is parallel, if $\nabla u=0$. Parallel fields are solutions to \eqref{ns} with $\beta=0$ and $p$ constant. 
    \item[(ii)] $u$ is harmonic, if $\Delta_H u=0$. Divergence free harmonic fields are solutions to \eqref{ns} with $\beta=-1$ and $p=-\tfrac{1}{2}\,g(u,u)$. 
    \item[(iii)] $u$ is Killing, if $\Sigma u=0$. Killing fields are solutions to \eqref{ns} with $\beta=1$ and $p=\tfrac{1}{2}\,g(u,u)$. 
\end{itemize}
As discussed in \cite{czubak,mj} different choices of $\beta$ can produce completely different solutions.  In some sense this is very natural, or at least intuitively plausible, because essentially choosing $\beta=-1$ corresponds to choosing antisymmetric part of $\Delta_B u$ while $\beta=1$ corresponds to choosing symmetric part of $\Delta_B u$.  So different choices of $\beta$ correspond to different physical models.  From the point of view of continuum mechanics $L$ would seem a natural choice, but perhaps in some situations other choices are appropriate. One might speculate that maybe in addition to above three values other values of $\beta$ could be relevant in some situations, or that $\beta$ might even be a function. Probably one should in any case require that $|\beta|\le 1$ since this condition guarantees that $-A$ is positive definite for divergence free vector fields.

\section{Restriction problem}
Our goal is to generalize the treatment of the restriction problem given in \cite{chczyo}. Let us first recall the setting of the problem in \cite{chczyo}. Let $D_x\subset \mathbb{R}^3$ be some convenient domain with the standard Riemannian metric $g_1$ and let us denote the coordinates of $D_x$ by $x$. Then let  $D_y\subset \mathbb{R}^3$ be some other domain with coordinates $y$ and let $\psi\,:\,D_y\to D_x$ be a diffeomorphism. Finally let $D_z\subset\mathbb{R}^2$ be a two dimensional domain with coordinates $z$ and let $f\,:\,D_z\to D_y$ be an embedding. Since all the analysis that follow is local it is not necessary to explicitly specify various domains. The metrics in $D_y$ and $D_z$ are obtained by pullbacks: $g_0=\psi^\ast g_1$ and $g= f^\ast g_0$.

 Of course one may view that $D_y$ and $D_x$ are same as manifolds, and $y$ and $x$ are merely two different coordinate systems. Indeed below when we consider the general case the ''intermediate'' coordinates are not needed. However, when computing concrete examples it is convenient to specify the problem in this way. 

Now we would like to compare the Laplacian defined on $D_y$ and the Laplacian on $M=f(D_z)$. Note that for divergence free vector fields the only difficulty is the Bochner part of the operator, so it is sufficient to analyse $\Delta_B$. Let $u$ be a vector field on $D_y$ which is tangent to $M$. Hence there is a vector field $v$ on $M$ such that $v=u$ on $M$. We want to compare $\Delta_Bu$ and $\Delta_B v$ somehow. The problem is that while  $\Delta_B v$ is a vector field on $M$ by definition, $\Delta_B u$ is not. 

Now let $p\in M$ and consider $T_pM$ as subspace of $T_pD_y$; then we can introduce the orthogonal projection $\pi\,:\, T_pD_y\to T_pM$ and we could compare $\Delta_B v$ and $\pi\big(\Delta_B u\big)$. Note that in \cite{chczyo} authors regarded $\Delta_B u$ as a one form and then pulled it back to $M$; however, this is equivalent to projection since if $X$ is any vector field then
\[
  f^\ast \flat\, \pi(X)= f^\ast \flat\, X\ .
\]
Let $\{b^1,b^2,b^3\}$ be an adapted orthonormal frame  and let $\{b_z^1,b_z^2\}$ be the corresponding   orthonormal frame on $M$. Now we set
\begin{equation}
    u=u^1b^1+u^2b^2+u^3b^3  
\quad\mathrm{and}\quad    
     v=v^1b_z^1+v^2b_z^2=f^\ast\big(  u^1b^1+u^2b^2\big)\ .
\label{kent}
\end{equation}
The derivatives of $u^k$ are denoted by subscripts:  $u^k_1=\partial u^k/\partial y_1$ etc. 
Since $u$ should be tangent to $M$ we require that $u^3$ is identically zero on $M$. Let us now consider the explicit problem in \cite{chczyo} before considering the general case. We feel that this is useful since it gives a good idea of the more complicated computations of the general case, and moreover it gives some indications about what the general result should look like. 

The following maps were chosen in \cite{chczyo}:
\begin{equation}
\begin{aligned}
    \psi(y)=&\big( ay_3\cos(y_1)\sin(y_2),  ay_3\sin(y_1)\sin(y_2), 
    y_3\cos(y_2)\big)\\
f(z)=&\big(z_1,z_2,1\big)
\end{aligned}
    \label{esim}
\end{equation}
Here $a$ is some positive constant. Then setting $\lambda=\sqrt{a^2\cos(y_2)^2+\sin(y_2)^2}$ we can choose for example the following frames:
\begin{equation}
\begin{aligned}
    b^1=&\frac{1}{ay_3\sin(y_2)}\,\partial_{y_1}&
    b^2=& \frac{1}{y_3\lambda}\,\partial_{y_2}&
    b^3=& \frac{(1-a^2)\sin(2y_2)}{2ay_3\lambda}\partial_{y_2}+
    \frac{\lambda}{a}\,\partial_{y_3}\\
    b_z^1=&\frac{1}{a\sin(z_2)}\,\partial_{z_1}&
    b_z^2=& \frac{1}{\lambda}\,\partial_{z_2}
\end{aligned}
   \label{kehys} 
\end{equation}
We will denote also the pullback $f^\ast \lambda$ by $\lambda$. 
Then the computations with formulas \eqref{toinen-perus} and \eqref{kaarevuudet} give
\[
\kappa=\frac{1}{\lambda^4}
\quad\mathrm{and}\quad
2\mathcal{H}=-\frac{1}{a\lambda}-
\frac{a}{\lambda^3} .
\]
The connection forms are
\begin{align*}
     \omega_{12}=& \frac{\cos(y_2)}{y_3\lambda\sin(y_2)}\,\theta^1&
       \omega_{13}=& \frac{1}{a y_3\lambda}\,\theta^1\\
         \omega_{23}=&
          \frac{a}{y_3\lambda^3}\,\theta^2+
          \frac{(1-a^2)\sin(2y_2)}{2y_3\lambda^3}\,\theta^3\ .
\end{align*}
Then we compute 
\[
  \pi\big(\Delta_B u\big)=c_1b^1+c_2b^2
  \quad\mathrm{where}\quad
  c_j=g_0(\Delta_Bu,b^j)\ .
\]
Now $c_j$ contain derivatives with respect to $y_3$ which are not meaningful intrinsically in $M$. Putting $\hat \lambda=\sqrt{a^2\sin(y_2)^2+\cos(y_2)^2}$ we have
\begin{align*}
 c_1=& \frac{1}{ a^{2} \sin \! (y_2 )^{2}}\,u^1_{11}+ 
\frac{\hat\lambda^2}{a^{2} }\,u^1_{22}+ 
\frac{\left(1-a^{2}\right) \sin \! \left(2y_2 \right) }{a^{2} }\,u^1_{23}+ \\
&\frac{\lambda^2}{a^{2}}\,u^1_{33}+\mathrm{lower\ order\ terms}  \,.
\end{align*}
The terms in $c_2$ are similar except that there are second order derivatives of $u^2$. 
So the goal is to somehow identify the terms containing derivatives with respect to $y_3$, and then see what is left. With the benefit of hindsight one might argue as follows: in $c_j$ there are at most second order derivatives, and all information about these is contained in the tensor $\overline{\nabla}\,\overline{\nabla} u$. However, we want to eliminate the derivatives in the normal direction so somehow it would be natural to consider $\overline{\nabla}\,\overline{\nabla} u\, b^3\otimes b^3$. Then we compute that 
\begin{align*}
  \pi\big(\overline{\nabla}\,\overline{\nabla} u\, b^3\otimes b^3\big)=&d_1b^1+d_2b^2
    \quad\mathrm{where}\quad
  d_j=g_0(\overline{\nabla}\,\overline{\nabla} u\, b^3\otimes b^3,b^j)
  \\ 
  c_1-d_1=  &\frac{\sin^{2}\left(2y_{2} \right) \left(a^2 -1\right)^{2} }{4\lambda^2 a^{2} }\,u^1_3+\mathrm{other\ terms}
\\
  c_2-d_2= &\frac{\sin^{2}\left(2y_{2} \right) \left(a^2 -1\right)^{2} }{4\lambda^2  a^{2} }\,u^2_3+\mathrm{other\ terms}
\end{align*}
where the other terms do not contain derivatives with respect to $y_3$. Hence the ''offending'' second order derivatives in $c_j$ have disappeared.  Somewhat unexpectedly the coefficients of the main terms in $c_j-d_j$  are equal. To get rid of these terms we try the Lie bracket;  this gives
\begin{align*}
\pi\big(\big[u, b^3\big] \big)=&e_1b^1+e_2b^2\\   
\hat v=&(c_1-d_1-2\mathcal{H}\, e_1)b^1+
(c_2-d_2-2\mathcal{H}\,e_2)b^2
\end{align*}
Now $\hat v$ contains only terms which make sense on $M$ so that we can consider the corresponding vector field on $M$, still denoted by $\hat v$. 
Of course it was not a priori clear that one must multiply  $e_j$ by   $2\mathcal{H}$, but once one tries the Lie bracket then it is easy to compute that this is the right choice. Now we can compare $\hat v$ to Bochner Laplacian, and simply computing we obtain that
\[
  \hat v=\Delta_B v+\kappa v\ .
\]
Let us summarize the results so far. In the problem specified by \eqref{esim} we have obtained using the adapted frame \eqref{kehys} that 
\begin{equation}
\pi\big(\Delta_B u\big)=
\Delta_B v+\kappa v+
\pi\Big(\overline{\nabla}\,\overline{\nabla} u\,b^3\otimes b^3+2\mathcal{H}[u,b^3]\Big)\ .
\label{tulos}
\end{equation}
It is curious that both Gaussian and mean curvature appear in the result. Now note that everything in the formula \eqref{tulos} makes sense on an arbitrary surface, so it is natural to assume that the result holds in general. This is indeed  almost true and one could actually prove it directly in a purely computational way, following the same steps as above. The word almost appears because in the general case there are terms which are automatically zero for the problem \eqref{esim}.  Anyway the symmetric Laplacian for divergence free fields $Lv=\Delta_B v+\kappa v$ will always be present.

However, it seems that it would be better give a more conceptual proof, and moreover it would be interesting to compute also the normal component. Also note that the properties of $\mathbb{R}^3$ were not really used; all computations were done in terms of the moving frame and these are  possible also when the ambient manifold is an  arbitrary three dimensional Riemannian manifold. Recall that all the tools introduced in section 2.1 are valid in this general case.
So instead of proving the formula \eqref{tulos} directly let us now move on to the general case.

\section{General formulation of the problem} 
Now we use the tools and notations outlined in section 2 to formulate the problem in the general case. Let us introduce vector fields $u$ and $v$ as in \eqref{kent}
such that $u^3$ is identically zero on $M$. Before analyzing the decomposition of $\Delta_Bu$ let us first discuss the divergence, bracket and curl. 

In the framework of our problem we may view $u$ as an extension of $v$, or $v$ as a restriction of $u$. Let us consider how the properties of $u$ and $v$ are related. Now in the flow problems the important situation is where the flow is incompressible and it would be natural to require that both $u$ and $v$ are incompressible. This is convenient to consider as an extension problem: given an incompressible flow $v$ can we find an incompressible extension $u$? In fact an even more general result is true. 
\begin{theorem}
Let $v$ be a vector field on $M$. Then there is a vector field $u$ which extends $v$ such that  
$\mathsf{div}(u)$ restricted to $M$ is equal to $\mathsf{div}(v)$.  
\label{div-lause}
\end{theorem}

We will call this kind of extension a compatible extension.

\begin{proof}
Note that  $
\mathsf{div}(b^1)\Xi_0=d\big(\theta^2 \wedge\theta^3\big)$ and similarly for other $b^j$ and $b_z^j$. Then   we compute, using the notation  \eqref{merk}, that  
\begin{align*}
    \mathsf{div}(u)=&
    u^1\mathsf{div}(b^1)+
    u^2\mathsf{div}(b^2)+
    u^3\mathsf{div}(b^3)+
    b^1(u^1)+b^2(u^2)
    +b^3(u^3)\\
    =& -(\gamma_2+\alpha_1) u^1+
    (\gamma_1-\alpha_2) u^2-
  (t_{11}+t_{22})    u^3+
    b^1(u^1)+b^2(u^2)
    +b^3(u^3)\\
     \mathsf{div}(v)=&
    v^1\mathsf{div}(b_z^1)+
    v^2\mathsf{div}(b_z^2)+
    b_z^1(v^1)+b_z^2(v^2)\\
    =&- \gamma_2 v^1+
 \gamma_1   v^2+
    b_z^1(v^1)+b_z^2(v^2)
\end{align*}
Hence $\mathsf{div}(u)$ restricted to $M$ is equal to $\mathsf{div}(v)$, i.e. $f^\ast \mathsf{div}(u)=\mathsf{div}(v)$, if
\begin{equation}
  b^3(u^3)=\alpha_1 u^1+\alpha_2 u^2\ .
\label{ehto}
\end{equation}
Now we can choose extensions $u^1$ and $u^2$ arbitrarily and if we require that \eqref{ehto} is satisfied in the neighborhood of $M$ then we have a linear first order PDE for $u^3$ which is noncharacteristic because $b^3$ is normal to $M$. Hence a  local solution with initial condition $u^3=0$ on $M$ exists  by standard theorems.
\end{proof}

Note that \eqref{ehto} is not the only way to extend $u^3$; one could for example also choose 
\[
  b^3(u^3)=\alpha_1 u^1+\alpha_2 u^2 
  +Cu^3
\]
where $C$ is a linear operator which contains only tangential derivatives.

In case of compressible problems it is perhaps not so clear how $\mathsf{div}(v)$ and $\mathsf{div}(u)$ should be related. However, it seems that in incompressible problems compatibility would be a natural requirement. Since the incompressible case is so fundamental we will in the sequel consider both general and compatible extensions. Note that if one thinks from the point of view of restrictions, then typically restrictions are not compatible.

The conclusion of the previous Theorem is still quite weak in the sense that if  $v$ is divergence free and  then the compatible extension constructed above is not necessarily  divergence free in the neighborhood of $M$. However, by modifying the argument slightly we obtain
\begin{corollary} Let $v$ be a vector field on $M$ such that $\mathsf{div}(v)=0$. Then there is a compatible extension $u$  such that  
$\mathsf{div}(u)=0$  in the neighborhood of $M$.  
\end{corollary}
\begin{proof}
We let $u^1$ and $u^2$ be $v^1$ and $v^2$ on $M$ and then we extend them in some way. Then the equation $\mathsf{div}(u)=0$, given above, can be written as
\begin{equation}
   b^3(u^3)- (t_{11}+t_{22})    u^3=F
\label{div0}
\end{equation}
where $F$ is some known function. This together with the initial condition $u^3=0$ on $M$ is a standard problem so the local solution exists.
\end{proof}

Let us now see how this looks with the example \eqref{esim}. Choosing for example $v^1=az_1/\lambda$ and $v^2=-z_2/\sin(z_2)$ we have $\mathsf{div}(v)=0$. Then putting simply $u^1=ay_1/\lambda$ and $u^2=-y_2/\sin(y_2)$ we get
\[
\frac{(1-a^2)\sin(2y_2)}{2ay_3\lambda}u^3_2+
    \frac{\lambda}{a}\,u^3_3+
   \frac{a^{2}+\lambda^{2}}{a\lambda^3 y_3}\, u^3=
   \frac{\left(a^2 -1\right) \cos \! \left(y_{2} \right)  y_{2} }{\lambda^3}\ .
\]
Let us then consider the vorticity of the flow and its restriction or extension. Let  $v$ be a vector field on $M$ and $u$ be a vector field on $M_0$; their vorticities are given by formulas
  \[
   \mathsf{rot}(v)=\ast d\, v\quad\mathrm{and}\quad
   \mathsf{curl}(u)=\sharp\ast d\, u
  \]
  where $\ast$ is the Hodge star.
\begin{lemma} Let $u$ be a vector field on $M_0$ which is tangent along $M$ and let $v$ be its restriction on $M$. Then $\mathsf{rot}(v)$ is essentially the pullback of $\mathsf{curl}(u)$:
\begin{equation}
  \mathsf{rot}(v)\Xi=f^\ast du=f^\ast\big(\ast\,\flat \,\mathsf{curl}(u)\big)\ .
  \label{rot-curl}
\end{equation}
In particular if $u$ is harmonic and divergence free, then $\mathsf{rot}(v)=0$.
\end{lemma}
\begin{proof}
This follows because  pullback commutes with exterior derivative: $df^\ast u=f^\ast du$. Then recall that if $u$ is harmonic and $\mathsf{div}(u)=-\delta u=0$ then also $du=0$. 
\end{proof}

If a planar flow is  interpreted as a three dimensional flow invariant with a certain direction, then the vorticity of the planar flow can be interpreted as a vector orthogonal to the plane. In the present context we could then ask if there is an extension of $v$ such that
\[
   \mathsf{curl}(u)=\mathsf{rot}(v)b^3\ .
\]
It is intuitively rather clear that all extensions do not satisfy this requirement. However,  the previous lemma implies that the normal component of $ \mathsf{curl}(u)$ is $\mathsf{rot}(v)b^3$. 
\begin{corollary}
  Let $v$ be a vector field on $M$ and let $u$ be its extension. Then $g_0( \mathsf{curl}(u),b^3)=\mathsf{rot}(v)$. 
 \label{rot-curl-koro}
\end{corollary}

One interesting class of flows is the Beltrami flow where vorticity and flow itself are linearly dependent, i.e. there is a function $\mu$ such that $\mathsf{curl}(u)=\mu u$. Hence if $u$ is tangent along $M$, then the vorticity is also tangent to $M$ and in this case  $\mathsf{rot}(v)=0$. If we require in addition that $v$ is divergence free, then $v$ would be in fact harmonic. 

\begin{proof}
 Recall that 
  \[
 \theta^3\wedge d u=
 g_0(\mathsf{curl}(u),b^3)\Xi_0\ .
 \]
Then using \eqref{rot-curl} we have
\[
( \theta^3\wedge d u)(b^1,b^2,b^3)=du(b^1,b^2)
 =dv(b_z^1,b_z^2)=\mathsf{rot}(v)\ . 
\] 
\end{proof}

Then it is natural to ask if there is an extension such that $ \mathsf{curl}(u)=\mathsf{rot}(v)b^3$. 
\begin{theorem}
 Let $v$ be a vector field on $M$. Then there is an extension $u$ such that
 \[
    \mathsf{curl}(u)=\mathsf{rot}(v)b^3\quad\mathrm{and}\quad
    \mathsf{div}(u)=\mathsf{div}(v)
    \quad\mathrm{on}\ M.
 \]
 Moreover, if $\mathsf{div}(v)=0$, then there is an extension such that $\mathsf{div}(u)=0$ in the neighborhood of $M$.
\end{theorem}
Note that a priori we have four equations for three unknowns, but since according to Corollary \ref{rot-curl-koro}   one condition is automatically satisfied, then it is rather natural to expect that there might be a  solution. 

\begin{proof}
Computing and evaluating on  $M$ we obtain
\begin{align*}
  \theta^1\wedge  d  u=&
\theta^1\wedge\Big(
    -u^1( \omega_{12}\wedge\theta^2+
    \omega_{13}\wedge\theta^3)-
    u^2\omega_{23}\wedge\theta^3+
    du^2\wedge\theta^2+
    du^3\wedge\theta^3
  \Big)\\
  \theta^2\wedge  d u=&
\theta^2\wedge\Big(
    -u^1 
    \omega_{13}\wedge\theta^3+
    u^2(\omega_{12}\wedge\theta^1-\omega_{23}\wedge\theta^3)+
    du^1\wedge\theta^1+
    du^3\wedge\theta^3
  \Big)\\
g_0(\mathsf{curl}(u),b^1)=&
  t_{12}u^1+t_{22}u^2 -du^2(b^3)\\
g_0(\mathsf{curl}(u),b^2)=&-t_{11}
u^1-t_{12}u^2 +du^1(b^3)
\end{align*}
This leads to a linear  system of first order PDE:
\begin{equation}
\begin{aligned}
   b^3(u^1)=& t_{11}u^1+t_{12}u^2 \\
    b^3(u^2)=& t_{12}u^1+t_{22}u^2 \\
     b^3(u^3)=&\alpha_1 u^1+\alpha_2 u^2 
\end{aligned}
\label{laajennus}
\end{equation}
By standard theorems there is a  local solution. 
Then if $\mathsf{div}(v)=0$ and we want that $\mathsf{div}(u)=0$ in the neighborhood of $M$ then we can simply replace the third equation in \eqref{laajennus} by the equation  \eqref{div0}. 
\end{proof}

Here also one gets other extensions by adding appropriate terms $C_ju^3$ where the operators $C_j$ contain derivatives only in tangential directions. In the case of example \eqref{esim} $t_{12}=0$ so that the new equations are decoupled:
\begin{align*}
  \frac{(1-a^2)\sin(2y_2)}{2ay_3\lambda}u^1_2+
    \frac{\lambda}{a}\,u^1_3=&
    -\frac{1}{a\lambda y_3}\,u^1\\
     \frac{(1-a^2)\sin(2y_2)}{2ay_3\lambda}u^2_2+
    \frac{\lambda}{a}\,u^2_3=&
    -\frac{a}{\lambda^3 y_3}\,u^2\ .
\end{align*}

Theorem \ref{div-lause} has also another interesting corollary.
\begin{corollary}
    If $\mathsf{div}(u)$ restricted to $M$ is equal to $\mathsf{div}(v)$ then $[u,b^3]$ and $\overline{\nabla}_{b^3}u  $ are tangent along $M$.
    \label{koro}
\end{corollary}
\begin{proof}
 Simply computing using \eqref{kova} we obtain
 \begin{align*}
   \overline{\nabla}_{b^3}u=&
   u^1 \overline{\nabla}_{b^3}b^1+
    u^2 \overline{\nabla}_{b^3}b^2+
     u^3 \overline{\nabla}_{b^3}b^3+
     b^3(u^1)b^1+ b^3(u^2)b^2+
      b^3(u^3)b^3\\
      =&\big(b^3(u^1)+\gamma_3u^2+\alpha_1u^3\big)b^1+
      \big(b^3(u^2) - \gamma_3u^1+\alpha_2u^3\big)b^2+
      \big(b^3(u^3)-\alpha_1u^1-\alpha_2u^2\big)b^3\ .
 \end{align*}
 Hence $ \overline{\nabla}_{b^3}u$ is tangent along $M$ if \eqref{ehto} holds. Since $ \overline{\nabla}_{u}b^3$ is always tangent along $M$ by  \eqref{wein} this implies that $[u,b^3]$ is then also tangent along $M$.
\end{proof}

In particular if the restriction/extension is compatible, then there is no need to project the bracket term in the formula \eqref{tulos}. However, the bracket, while seemingly natural here, does not appear in the decomposition below because it is necessarily extrinsic to $M$. Hence we must in any case analyze separately $\overline{\nabla}_{b^3}u$ and $\overline{\nabla}_u b^3$.

Before starting the decomposition let us first establish some notation which will facilitate the computations below. The main issue is how to define  $\overline{\nabla}_{b^3}u$ in the neighborhood of $M$. We will consider two cases: (1) the general case where the extension of $u$ is arbitrary, and (2) divergence free extension where we extend $u^3$ such that $\mathsf{div}(u)=0$ in the neighborhood of $M$.

Both cases can then be divided to two subcases, depending if we want the extension/restriction to be compatible or not. Let us then set $X_3= \overline{\nabla}_{b^3}b^3=\alpha_1b^1+\alpha_2b^2$ and 
\[
  \rho=b^3(u^3)-\alpha_1 u^1-\alpha_2 u^2=
  b^3(u^3)-g_0(X_3,u)\  .
\]
The vector field $X_3$ will appear frequently in various formulas. When we evaluate it on $M$ we can also write $X_3=\alpha_1b_z^1+\alpha_2b_z^2$, without changing the notation. Let us also introduce $q=b^3(u^1)b^1+
    b^3(u^2)b^2$.
With this notation one can write
\begin{equation}
  \overline{\nabla}_{b^3}u=
  q+u^3X_3+\gamma_3(u^2b^1-u^1b^2)+\rho b^3
\label{dub3}
\end{equation}
in the neighborhood of $M$.
In the compatible case $\rho=0$ on $M$, but not (necessarily) in the neighborhood of $M$.  Then we compute using \eqref{kaare} and \eqref{omega-kaava} that
\begin{equation}
\begin{aligned}
b^1(t_{12})=&b^2(t_{11})-\Omega_{13}(b^1,b^2)+
   (t_{11}-t_{22})\gamma_1+2t_{12}\gamma_2\\
   b^2(t_{12})=&b^1(t_{22})+\Omega_{23}(b^1,b^2)
  -2t_{12}\gamma_1+ (t_{11}-t_{22})\gamma_2
\end{aligned}
\label{t-der}
\end{equation}

\section{Decomposition of Bochner Laplacian}
We are now ready to start the decomposition of $\Delta_Bu$. The basis of the decomposition is the identity
\begin{align*}
\Delta_Bu=&\overline{\nabla}\,\overline{\nabla} u
 \Big(b^1\otimes b^1+b^2\otimes b^2+b^3\otimes b^3\Big)\\ =&\overline{\nabla}_{b^1}\overline{\nabla}_{b^1}u+
\overline{\nabla}_{b^2}\overline{\nabla}_{b^2}u
  -\overline{\nabla}_{X}u+\overline{\nabla}\,\overline{\nabla} u\,
 b^3\otimes b^3
\end{align*}
where  $X=X_1+X_2$ and $X_j=\overline{\nabla}_{b^j}b^j$. Note that we first compute with vector field $u$ and then evaluate the final result on $M$ so in the end we have the vector $v$ and its derivatives, but also terms containing the normal derivatives of $u$.

Let us first 
 consider the term $\overline{\nabla}_{X}u$. 
\begin{lemma} $\overline{\nabla}_{X}u$ on $M$ is given by
\[
   \overline{\nabla}_{X}u=  \nabla_{Kw} v+ 2\mathcal{H}\,(q+\gamma_3 Kv)+\big(2\mathcal{H}\, \rho+S(Kw,v)\big)b^3  
\]
where $ K=\theta_z^1\otimes\theta_z^2-\theta_z^2\otimes\theta_z^1$  and $w=\gamma_1b_z^1+\gamma_2b_z^2$.
\end{lemma}
\begin{proof}
On $M$ we can write $X_j=Y_j+t_{jj}b^3$ where $Y_j=\nabla_{b_z^j}b_z^j$ so that  using \eqref{dub3} we obtain
\[
   \overline{\nabla}_{X_j}u=
    \nabla_{Y_j}v+S(Y_j,v)b^3+
   t_{jj} \overline{\nabla}_{b^3}u=
    \nabla_{Y_j}v+S(Y_j,v)b^3
   + t_{jj}(q+\gamma_3Kv+\rho b^3)\ .
\]
Since $Kw=Y_1+Y_2$,  adding gives the result.
\end{proof}

It turns out to be useful to define  tensors 
\begin{align*}
   T_0=&\Big(2\,b_z^1\big(\mathcal{H}\big)+\Omega_{23}(b_z^1,b_z^2)\Big)\theta_z^1+
\Big(  2\, b_z^2\big(\mathcal{H}\big)-
\Omega_{13}(b_z^1,b_z^2)\Big)\theta_z^2\quad\mathrm{and}\\
    Q=& b_z^1(v^1)b_z^1\otimes b_z^1+
     b_z^1(v^2)b_z^1\otimes b_z^2+
      b_z^2(v^1)b_z^2\otimes b_z^1+
       b_z^2(v^2)b_z^2\otimes b_z^2
\end{align*}
Recall that for flat manifolds $\Omega_{ij}=0$. Since $2\mathcal{H}$ is the trace of $S$ one could also express $2\,b_z^1\big(\mathcal{H}\big)\theta_z^1+
  2\, b_z^2\big(\mathcal{H}\big)\theta_z^2$ using $\nabla S$; however, this does not seem to make the formulas any simpler.  
\begin{lemma} Let $\sigma=S_{ij}Q^{ij}$; then 
\begin{align*}
\overline{\nabla}_{b^1}\overline{\nabla}_{b^1}u+
\overline{\nabla}_{b^2}\overline{\nabla}_{b^2}u=&
 \Delta_Bv+\kappa\,v-2\mathcal{H}\,Sv+\nabla_{Kw}v+\\
&\Big(2\,\sigma+T_0v+S(Kw,v)
+2S(w,Kv)\Big)b^3\ .
\end{align*}
\end{lemma}

Somewhat strangely the tensor $P=\kappa\,g-2\mathcal{H}\,S$ which appears above has a name: $-P$ is the \emph{third} fundamental form. Note that $P$ is negative definite. 

\begin{proof}
We have first
\[
  \overline{\nabla}_{b^1}u=
  \nabla_{b_z^1}v+S(v,b_z^1)b^3\ .
\]
Then differentiating the tangential component gives
\[ 
\overline{\nabla}_{b^1}\nabla_{b_z^1}u=
\nabla_{b_z^1}\nabla_{b_z^1}v+S(\nabla_{b_z^1}v,b_z^1)b^3=
  \nabla\nabla v \, b_z^1\otimes b_z^1+\nabla_{Y_1}v  
  +S(\nabla_{b_z^1}v,b_z^1)b^3\ .
\]
Computing similarly with $\overline{\nabla}_{b^2}\nabla_{b_z^2}u$ and adding we obtain
\begin{align*}
 \overline{\nabla}_{b^1}\nabla_{b^1}u+
 \overline{\nabla}_{b^2}\nabla_{b^2}u=&
 \Delta_Bv+\nabla_{Kw}v
  +\Big(S(\nabla_{b_z^1}v,b_z^1)
  +S(\nabla_{b_z^2}v,b_z^2)\Big)b^3
\end{align*}
Then it is straightforward to check that
\[
S(\nabla_{b_z^1}v,b_z^1)+S(\nabla_{b_z^2}v,b_z^2)=
\sigma+S(w,Kv)\ .
\]
Then from the normal component we get using \eqref{t-der} that
\[
b^1_z(S(v,b_z^1))+b^2_z(S(v,b_z^2))=
\sigma+T_0v+S(Kw,v)+S(w,Kv)\ .
\]
Then
\begin{align*}
\sum_{j=1}^2
S(v,b_z^j)\overline{\nabla}_{b^j}b^3
=&-(t_{11}v^1+t_{12}v^2)(t_{11}b_z^1+t_{12}b_z^2)
-(t_{12}v^1+t_{22}v^2)(t_{12}b_z^1+t_{22}b_z^2)\\
=&(t_{11}t_{22}-t_{12}^2)v+
(t_{11}+t_{22})\Big((t_{11}v^1+t_{12}v^2)b_z^1+
(t_{12}v^1+t_{22}v^2)b_z^2\Big)\\
=&\kappa\,v-2\mathcal{H}\,Sv\ .
\end{align*}
Now adding everything gives the result.
\end{proof}

Putting together previous Lemmas gives
\begin{equation}
\begin{aligned}
  \overline{\nabla}\,\overline{\nabla} u
 \Big(b^1\otimes b^1+b^2\otimes b^2\Big)=&
 \Delta_Bv+\kappa\,v-2\mathcal{H}\Big(
 \,Sv +q+\gamma_3 Kv\Big)+\\
 &\Big(2\,\sigma+T_0v+2S(w,Kv)
-2\mathcal{H}\, \rho \Big)b^3
\end{aligned}
\label{lemmat12}
\end{equation}

Then we have to analyze
\[
\overline{\nabla}\,\overline{\nabla} u\,
 b^3\otimes b^3=
 \overline{\nabla}_{b^3}\overline{\nabla}_{b^3}u-\overline{\nabla}_{X_3}u=
  \overline{\nabla}_{b^3}\overline{\nabla}_{b^3}u-\nabla_{X_3}v-S(X_3,v)b^3\ .
\]
Let us first set
\begin{align*}
  T_1=&b^3(\alpha_1)\theta_z^1+
  b^3(\alpha_2)\theta_z^2&
  T_2= & \gamma_3\, KX_3-T_1\\  
 E_0=&\flat X_3\otimes\flat X_3  &
  E_1=&E_0+b^3(\gamma_3)K-
  \gamma_3^2g\ .
\end{align*}

\begin{lemma} Let  $\overline{\nabla}_{b^3}u$ be given in a neighborhood of $M$ by \eqref{dub3}. 
Then $\overline{\nabla}_{b^3}\overline{\nabla}_{b^3}u$ evaluated on $M$ is given by
\begin{align*}
\overline{\nabla}_{b^3}\overline{\nabla}_{b^3}u=& \Big( \overline{\nabla}\, \overline{\nabla}u^1\ b^3\otimes b^3+X_3(v^1) \Big)b^1+
  \Big( \overline{\nabla}\, \overline{\nabla}u^2\ b^3\otimes b^3 +X_3(v^2)\Big)b^2+
  2\gamma_3Kq+\\
  &2\rho X_3+E_1v+
 \Big(\overline{\nabla}\, \overline{\nabla}u^3\ b^3\otimes b^3+T_2v-2g(X_3,q)\Big)b^3\ .
\end{align*}
\end{lemma}
\begin{proof}  We compute
\[
 \overline{\nabla}_{b^3}b^3(u^1)b^1=
  b^3(u^1) \overline{\nabla}_{b^3} b^1+
  b^3\big( b^3(u^1)\big)b^1=
- \gamma_3b^3(u^1)b^2 - \alpha_1 b^3(u^1)b^3+
  \Big( \overline{\nabla}\, \overline{\nabla}u^1\ b^3\otimes b^3+X_3(v^1)\Big)b^1\ .
\]
Computing similarly with $ \overline{\nabla}_{b^3}b^3(u^2)b^2$ and adding gives 
\[
\overline{\nabla}_{b^3}q=
\Big( \overline{\nabla}\, \overline{\nabla}u^1\ b^3\otimes b^3+X_3(v^1) \Big)b^1+
  \Big( \overline{\nabla}\, \overline{\nabla}u^2\ b^3\otimes b^3 +X_3(v^2)\Big)b^2+
  \gamma_3Kq-g(X_3,q)b^3\ .
\]
 Then $\overline{\nabla}_{b^3}u^3X_3=b^3(u^3)X_3=\rho X_3+E_0v$ because we evaluate on $M$. Next we have 
\begin{align*}
   \overline{\nabla}_{b^3}\gamma_3(u^2b^1-u^1b^2)=&b^3(\gamma_3)Kv+
  \gamma_3\big(  Kq+u^2 \overline{\nabla}_{b^3}b^1-u^1 \overline{\nabla}_{b^3}b^2\big)\\ =&\gamma_3 Kq+b^3(\gamma_3)Kv-\gamma_3^2v+ \gamma_3 g(KX_3,v)b^3\ .
\end{align*}
Finally we compute
\begin{equation}
\begin{aligned}
\overline{\nabla}_{b^3}\rho b^3=&
  b^3(\rho)b^3+\rho X_3
  \quad\mathrm{where}\\
  b^3(\rho)=&\overline{\nabla}\, \overline{\nabla}u^3\ b^3\otimes b^3-
   b^3(\alpha_1)v^1-
  b^3(\alpha_2)v^2-g(X_3,q)\\
  =&\overline{\nabla}\, \overline{\nabla}u^3\ b^3\otimes b^3-
      T_1v-g(X_3,q)     \ .
\end{aligned}
\label{tekn-kaava}
\end{equation}
\end{proof}

This gives immediately
\begin{lemma}

Putting $E_2=E_1-g(X_3,w)K$ we have
\begin{align*}
\overline{\nabla}\,\overline{\nabla} u\,
 b^3\otimes b^3=& \Big( \overline{\nabla}\, \overline{\nabla}u^1\ b^3\otimes b^3 \Big)b^1+
  \Big( \overline{\nabla}\, \overline{\nabla}u^2\ b^3\otimes b^3 \Big)b^2+
  2\gamma_3Kq+\\
  &2\rho X_3+E_2v+
 \Big(\overline{\nabla}\, \overline{\nabla}u^3\ b^3\otimes b^3+T_2v-2g(X_3,q)-S(X_3,v)\Big)b^3\ .
\end{align*}
\label{normaali-lemma}
\end{lemma}
\begin{proof}
One has to check that 
\[
   X_3(v^1)b^1+X_3(v^2)b^2-\nabla_{X_3}v=
   -g(X_3,w)Kv\ .
\]
\end{proof}

This gives
\begin{theorem}
Let us write $\Delta_Bu=B_t+B_nb^3$ where $B_t$ is the tangential component and let  $T=T_0+T_2$, 
\[
  N=2\gamma_3K-2\mathcal{H}\,g
  \quad\mathrm{and}\quad
  E=E_2-2\mathcal{H}\,\gamma_3K\ .
\]
Then we have
\begin{equation}
\begin{aligned}
  B_t=&\Delta_Bv+\kappa\,v-2\mathcal{H}\,Sv  +Ev+2\rho X_3+\\
  &\Big( \overline{\nabla}\, \overline{\nabla}u^1\ b^3\otimes b^3 \Big)b^1+
  \Big( \overline{\nabla}\, \overline{\nabla}u^2\ b^3\otimes b^3 \Big)b^2+Nq\\
  B_n=&2\,\sigma+2S(w,Kv)-S(X_3,v)+Tv+ \overline{\nabla}\, \overline{\nabla}u^3\ b^3\otimes b^3    -2\,g(X_3,q)-2\mathcal{H}\,\rho \ .
\end{aligned}
\label{yl-kaavat}
\end{equation}
In the compatible case we have $\rho=0$.
\label{lause}
\end{theorem}
Note that the sign of $S$ and $\mathcal{H}$ depend on the choice of orientation of $b^3$, but the sign of $B_nb^3$ does not depend on this choice. 

\begin{proof}
We simply combine the formula \eqref{lemmat12} and Lemma \ref{normaali-lemma}. 
\end{proof}

In the divergence free case this yields
\begin{theorem}
Let us set $\ell=b^3(\gamma_1)b_z^1+b^3(\gamma_2)b_z^2$.
 If $\mathsf{div}(u)=0$ in the neighborhood of $M$ we obtain
\begin{equation}
\begin{aligned}
  B_t=&\Delta_Bv+\kappa\,v-2\mathcal{H}\,Sv  +Ev-2\,\mathsf{div}(v) X_3+\\
  &\Big( \overline{\nabla}\, \overline{\nabla}u^1\ b^3\otimes b^3 \Big)b^1+
  \Big( \overline{\nabla}\, \overline{\nabla}u^2\ b^3\otimes b^3 \Big)b^2+Nq\\
  B_n=&2\,\sigma+\gamma_3 \big(b_z^2(v^1)-b_z^1(v^2)\big)+
  2S(w,Kv)+S_{\mathsf{adj}}(X_3,v)+T_{\mathsf{div}}v\\
&  -  \overline{\nabla}\,\overline{\nabla} u^1\,(  b^1\otimes b^3)-
\overline{\nabla}\,\overline{\nabla} u^2\,(  b^2\otimes b^3)
 +g(Kw,q)
\end{aligned}
\label{div0-deko}
\end{equation}
where 
\[
   T_{\mathsf{div}}=T_0+K\big(\ell+
   \gamma_3X_3\big)\ .
\]
 In the compatible case one just substitutes $\mathsf{div}(v)=0$ to $B_t$ and for $B_n$ one obtains
 \begin{align*}
   B_n=& 2\hat\sigma+\gamma_3\big(b^2(v^1)-b^1(v^2)\big)
+  2S(w,Kv)+S_{\mathsf{adj}}(X_3,v)+T_{\mathsf{div,c}}v\\
&  -  \overline{\nabla}\,\overline{\nabla} u^1\,(  b^1\otimes b^3)-
\overline{\nabla}\,\overline{\nabla} u^2\,(  b^2\otimes b^3)
 +g(Kw,q)
 \end{align*}
 where 
\begin{align*}
\hat \sigma=&(t_{22}-t_{11})b^2(v^2)
+t_{12}\big(b^2(v^1)
+b^1(v^2)\big)
\quad\mathrm{and}\\
   T_{\mathsf{div,c}}=&T_0+K\big(\ell+
   \gamma_3X_3+2t_{11}w\big)\ .
\end{align*}

\label{lause-div0}
\end{theorem}

\begin{proof}
Recall that if $\mathsf{div}(u)=0$ then
\[
  b^3(u^3)=-b^1(u^1)-b^2(u^2)+(t_{11}+t_{22})u^3+
 \alpha_1u^1+\alpha_2u^2+\gamma_2u^1-\gamma_1u^2\ .
\]
Hence on $M$ we have $
  \rho X_3=-\mathsf{div}(v)X_3$ 
which gives the formula for $B_t$. Then we have to compute 
\[
\overline{\nabla}\, \overline{\nabla}u^3\ b^3\otimes b^3=b^3(b^3(u^3))=
b^3\Big(-b^1(u^1)-b^2(u^2)+(t_{11}+t_{22})u^3+
  \alpha_1u^1+\alpha_2u^2+\gamma_2u^1-\gamma_1u^2\Big)\ .
\]
We compute first
\begin{align*}
-b^3\big(b^1(u^1)+b^2(u^2)\big)=&
-\overline{\nabla}\,\overline{\nabla} u^1\,(  b^1\otimes b^3)-
\overline{\nabla}\,\overline{\nabla} u^2\,(  b^2\otimes b^3)+
\gamma_3\,\big(b_z^2(v^1)-b_z^1(v^2)\big)+g(X_3,q)\\
b^3\big((t_{11}+t_{22})u^3\big)=&
 2\mathcal{H}\, b^3(u^3)=2\mathcal{H}\,(\rho+g(X_3,v))\ .
\end{align*}
Note that $2\mathcal{H}\, \rho$ cancels directly the corresponding term in \eqref{yl-kaavat} and moreover
\[
2\mathcal{H}\,g(X_3,v)-S(X_3,v)=
S_{\mathsf{adj}}(X_3,v)\ .
\]
Then using \eqref{tekn-kaava} we have
\[
 b^3(\alpha_1u^1+\alpha_2u^2)=  g(X_3,q)+T_1v\ .
 \]
 Finally
 \[
 b^3(\gamma_2 u^1-\gamma_1 u^2)=
     g(Kw,q)+g(K\ell,v)\ .
 \]
If $\mathsf{div}(v)=0$ we can solve
\[
   b^1(v^1)=g(Kw,v)-b^2(v^2)
\]
which gives
\[
\sigma=(t_{22}-t_{11})b^2(v^2)
+t_{12}\big(b^2(v^1)
+b^1(v^2)\big)+t_{11}g(Kw,v)\ .
\]
\end{proof}

In \cite[Theorem 1.1]{chczyo} $\flat B_t$ was computed for the problem \eqref{esim} in the divergence free and compatible case.

There are several interesting points about the formulas  in above Theorems.  For definiteness let us concentrate on Theorem \ref{lause-div0}, but  similar remarks of course apply to formulas in Theorem \ref{lause}. First it is clear that $B_t$ can be rather arbitrary. For example one may view 
\begin{align*}
& \overline{\nabla}\, \overline{\nabla}u^1\ b^3\otimes b^3 -2\mathcal{H}\,b^3(u^1)+
  2\gamma_3b^3(u^2)=h_1\\
  & \overline{\nabla}\, \overline{\nabla}u^2\ b^3\otimes b^3 
  -2\mathcal{H}\,b^3(u^2)
  -2\gamma_3b^3(u^1)=h_2
\end{align*}
as a system of PDE for $u^1$ and $u^2$ in the neighborhood of $M$ and solve it, given some $h_j$ and some appropriate boundary conditions. Even on a more elementary level it is clear that choosing extensions appropriately one can make the terms $\overline{\nabla}\, \overline{\nabla}u^j\ b^3\otimes b^3$ arbitrarily big compared to $\Delta_B v$.  So the overall conclusion is that  $\pi(\Delta_B u)$ and $\Delta_B v$ are unrelated in general. To have a meaningful relationship the normal derivatives of $u^j$ should be ''small'', in other words only the first lines in the formulas for $B_t$ and $B_n$ in \eqref{div0-deko} should be ''big''. In this case we could write
\begin{align*}
 B_t\approx&\,\Delta_Bv+\kappa\,v-2\mathcal{H}\,Sv+Ev
 -2 \mathsf{div}(v)X_3\\
  B_n\approx&\,2\,\sigma+\gamma_3  \big(b_z^2(v^1)-b_z^1(v^2)\big)+
  2S(w,Kv)+S_{\mathsf{adj}}(X_3,v)+T_{\mathsf{div}}v\ .
\end{align*}
The symmetric Laplacian $Lv=\Delta_B v+\kappa\,v$ (for divergence free $v$) appears in some sense naturally in the formula of $B_t$. Note that the operator $L$ is intrinsically defined. The tensor $2\mathcal{H}\,S$ depends on the embedding, but this is perfectly fine in the present context;  the important thing is that it does not depend on the chosen frame. The presence of mean curvature in the formulas is  intriguing. It  suggests that minimal surfaces are somehow special here, but it is not clear how this should or could be interpreted from the point of view of physical model.

However, unfortunately the tensor $E$ does  depend on the chosen frame. Let us simply illustrate the difficulties without resolving them. 
Let us   choose $a=1$ in \eqref{esim} and compare our computations to the ones in \cite{chczdi}. In this case we have $\kappa=1$ and $X_3=\gamma_3=0$. Choosing further that $u^3$ is identically zero and $u^1$ and $u^2$ do not depend on $y_3$ one obtains
\[
  \pi(\Delta_B u)=
  \Delta_Bv+\kappa v-2\mathcal{H}\,S v=
  \Delta_B v-v=\Delta_H v\ .
\]
In \cite[p. 344]{chczdi} the authors obtained using a  similar extension $\hat u=u^1\partial_{y_1}+u^2\partial_{y_2}$  that 
\[
    \pi(\Delta_B \hat u)=\Delta_B \hat v+\hat v=L \hat v\ .
\]
The authors considered the divergence free case, but one obtains the same result also in general, because the extensions considered here are in any case compatible. 
The above results are not  contradictory because in our case $u=u^1b^1+u^2b^2$ and $b^j$ do depend on $y_3$ although $u^j$ do not. Also the frame $\{\partial_{y_1},\partial_{y_2},\partial_{y_3}\}$, used in \cite{chczdi,chczyo}, is not orthonormal. However, this illustrates clearly how seemingly minor differences in the extension or the frame can have dramatic consequences for the restriction even in this simple case.

Let us then give an example how the tensor $E$ depends on the frame even in the case that the frames are orthonormal. 
  Writing $E$ explicitly gives
\[
E= \flat X_3\otimes\flat X_3 -
  \gamma_3^2g
   +\big(b^3(\gamma_3)-2\mathcal{H}\,\gamma_3-g(X_3,w) \big)K
\ .
\]
Note that $w$ is in fact well-defined: one way to see this is to observe that  $w=[b^1_z,b^2_z]$. The problem is with $\gamma_3$ and $X_3$.
Let us then for  the final time consider the problem \eqref{esim}.  With the frame given in \eqref{kehys} we have $\gamma_3=g(X_3,w)=0$ and 
\[
   E= \flat X_3\otimes\flat X_3=
   \frac{(1-a^2)^2\sin^2(2z_2)}{4\lambda^6}\,\theta_z^2\otimes \theta_z^2\ .
\] 
Let us then choose  the following frame:
\begin{align*}
  \tilde b^1=&\frac{1}{\sqrt{\mu_0}}\,\partial_{y_1}+
  \frac{y_3-1}{\sqrt{\mu_0}}\,\partial_{y_3}\\
   \tilde b^2=&\frac{(1-a^2)(y_3-1)\sin(2y_2)}{2a\sqrt{\mu_0\mu_1}}\,\partial_{y_1}+
  \frac{\sqrt{\mu_0}}{ay_3\sqrt{\mu_1}}\,\partial_{y_2}+
  \frac{(1-a^2)(y_3-1)^2\sin(2y_2)}{2a\sqrt{\mu_0\mu_1}}\,\partial_{y_3}\\
  \tilde b^3=&-\frac{y_3-1}{ay_3\sin(y_2)\sqrt{\mu_1}}\,\partial_{y_1}+
  \frac{(1-a^2)\sin(2y_2)\sin(y_2)}{2a\sqrt{\mu_1}}\,\partial_{y_2}+
  \frac{y_3\lambda^2\sin(y_2)}{a\sqrt{\mu_1}}\,\partial_{y_3}\\
  \mu_0=&(y_3-1)^2\hat\lambda^2+
  a^2y_3^2\sin^2(y_2)\\
  \mu_1=&(y_3-1)^2+y_3^2\lambda^2\sin^2(y_2)
\end{align*}
Note that evaluated on $M$  we have $\tilde b^j=b^j$  so that $b^j_z$ and hence $\theta^j_z$ are the same as before. However, when we simply compute we obtain that 
\[
\gamma_3=\frac{(1-a^2)\cos(z_2)}{a^2}
\quad\mathrm{and}\quad
\alpha_1\gamma_1+\alpha_2\gamma_2=
g(X_3,w)=
-\frac{\cos(z_2)}{a\lambda\sin^2(z_2)}
\]
when we evaluate on $M$ and further
\begin{align*}
  E=&\frac{\lambda^4-(1+a^2)\lambda^2+2a^2}{a^4\sin^2(z_2)}\,\theta_z^1\otimes \theta_z^1+
  \frac{(a^2-1)\cos(z_2)(3\lambda^2-a^2)}{a^3\lambda}\, \theta_z^1\otimes \theta_z^2+\\
& \frac{(a^2-1)\cos(z_2)(2a^2+\lambda^2a^2-3\lambda^4)}{a^3\lambda^3}\, \theta_z^2\otimes \theta_z^1+
 \frac{(\lambda^2-1)\big((1-a^2)\lambda^6-a^4\lambda^2+a^6\big)}{a^4\lambda^6}\, \theta_z^2\otimes \theta_z^2\ .
\end{align*}
Similarly the normal component $B_nb^3$ depends on the chosen frame. Incidentally it seems that the properties of the normal component has not been analyzed previously. One would think that the normal component is also important when assessing the relationship between the three dimensional flow and its restriction. 

In conclusion we always have
\[
   \pi\big(\Delta_Bu\big)=
   \Delta_Bv+\kappa v-2\mathcal{H}\,Sv+
   \ \mathrm{other\ terms}
\]
where the other terms depend on the extension of $v$ and the frame. While it remains unclear how this frame dependence might be resolved, Theorems \ref{lause} and \ref{lause-div0} provide general and convenient  formulas for exploring this problem further.

\printbibliography

\end{document}